\documentclass[reqno, 12pt]{amsart}
\pdfoutput=1
\makeatletter
\let\origsection=\section \def\section{\@ifstar{\origsection*}{\mysection}}
\def\mysection{\@startsection{section}{1}\z@{.7\linespacing\@plus\linespacing}{.5\linespacing}{\normalfont\scshape\centering\S}}
\makeatother

\usepackage{amsmath,amssymb,amsthm}
\usepackage{mathrsfs}
\usepackage{mathabx}\changenotsign
\usepackage{dsfont}
\usepackage{bbm}

\usepackage{xcolor}
\usepackage[backref=section]{hyperref}
\usepackage[ocgcolorlinks]{ocgx2}
\hypersetup{
	colorlinks=true,
	linkcolor={red!60!black},
	citecolor={green!60!black},
	urlcolor={blue!60!black},
}

\usepackage[open,openlevel=2,atend]{bookmark}

\usepackage[abbrev,msc-links,backrefs]{amsrefs}
\usepackage{doi}

\renewcommand{\PrintDOI}[1]{\doi{#1}}

\usepackage[T1]{fontenc}
\usepackage{lmodern}
\usepackage[babel]{microtype}
\usepackage[english]{babel}

\usepackage{geometry}
\numberwithin{equation}{section}
\numberwithin{figure}{section}

\usepackage{enumitem}

\let\polishlcross=\l
\def\l{\ifmmode\ell\else\polishlcross\fi}

\let\emptyset=\varnothing
\let\setminus=\smallsetminus

\makeatletter
\def\moverlay{\mathpalette\mov@rlay}
\def\mov@rlay#1#2{\leavevmode\vtop{   \baselineskip\z@skip \lineskiplimit-\maxdimen
		\ialign{\hfil$\m@th#1##$\hfil\cr#2\crcr}}}
\newcommand{\charfusion}[3][\mathord]{
	#1{\ifx#1\mathop\vphantom{#2}\fi
		\mathpalette\mov@rlay{#2\cr#3}
	}
	\ifx#1\mathop\expandafter\displaylimits\fi}
\makeatother

\DeclareFontFamily{U}  {MnSymbolC}{}
\DeclareSymbolFont{MnSyC}         {U}  {MnSymbolC}{m}{n}
\DeclareFontShape{U}{MnSymbolC}{m}{n}{
	<-6>  MnSymbolC5
	<6-7>  MnSymbolC6
	<7-8>  MnSymbolC7
	<8-9>  MnSymbolC8
	<9-10> MnSymbolC9
	<10-12> MnSymbolC10
	<12->   MnSymbolC12}{}
\DeclareMathSymbol{\powerset}{\mathord}{MnSyC}{180}

\usepackage{tikz}
\usetikzlibrary{calc,decorations.pathmorphing}
\usetikzlibrary{arrows,decorations.pathreplacing}
\pgfdeclarelayer{background}
\pgfdeclarelayer{foreground}
\pgfdeclarelayer{front}
\pgfsetlayers{background,main,foreground,front}

\usepackage{multicol}
\usepackage{subcaption}
\newcommand{\pedge}[9]{
	
	\ifx\relax#6\relax
	\def\qoffs{0pt}
	\else
	\def\qoffs{#6}
	\fi
	
	\def\phedge{
		($#1+#5!\qoffs!-90:#2-#5$) -- 
		($#2+#1!\qoffs!-90:#3-#1$) -- 
		($#3+#2!\qoffs!-90:#4-#2$) -- 
		($#4+#3!\qoffs!-90:#5-#3$) -- 
		($#5+#4!\qoffs!-90:#1-#4$) -- cycle}

	\coordinate (12) at ($#1!\qoffs!90:#2$);
	\coordinate (15) at ($#1!\qoffs!-90:#5$);
	\coordinate (23) at ($#2!\qoffs!90:#3$);
	\coordinate (21) at ($#2!\qoffs!-90:#1$);
	\coordinate (34) at ($#3!\qoffs!90:#4$);
	\coordinate (32) at ($#3!\qoffs!-90:#2$);
	\coordinate (45) at ($#4!\qoffs!90:#5$);
	\coordinate (43) at ($#4!\qoffs!-90:#3$);
	\coordinate (51) at ($#5!\qoffs!90:#1$);
	\coordinate (54) at ($#5!\qoffs!-90:#4$);

	\def\nphedge{
		(15) let \p1=($(15)-#1$), \p2=($(12)-#1$) in 
		arc[start angle={atan2(\y1,\x1)}, delta angle={atan2(\y2,\x2)-atan2(\y1,\x1)-360*(atan2(\y2,\x2)-atan2(\y1,\x1)>0)}, x radius=\qoffs, y radius=\qoffs] --
		(21) let \p1=($(21)-#2$), \p2=($(23)-#2$) in 
		arc[start angle={atan2(\y1,\x1)}, delta angle={atan2(\y2,\x2)-atan2(\y1,\x1)-360*(atan2(\y2,\x2)-atan2(\y1,\x1)>0)}, x radius=\qoffs, y radius=\qoffs] --
		(32) let \p1=($(32)-#3$), \p2=($(34)-#3$) in 
		arc[start angle={atan2(\y1,\x1)}, delta angle={atan2(\y2,\x2)-atan2(\y1,\x1)-360*(atan2(\y2,\x2)-atan2(\y1,\x1)>0)}, x radius=\qoffs, y radius=\qoffs] --
		(43) let \p1=($(43)-#4$), \p2=($(45)-#4$) in 
		arc[start angle={atan2(\y1,\x1)}, delta angle={atan2(\y2,\x2)-atan2(\y1,\x1)-360*(atan2(\y2,\x2)-atan2(\y1,\x1)>0)}, x radius=\qoffs, y radius=\qoffs] --
		(54) let \p1=($(54)-#5$), \p2=($(51)-#5$) in 
		arc[start angle={atan2(\y1,\x1)}, delta angle={atan2(\y2,\x2)-atan2(\y1,\x1)-360*(atan2(\y2,\x2)-atan2(\y1,\x1)>0)}, x radius=\qoffs, y radius=\qoffs] --
		cycle}

	\ifx\relax#7\relax
	\def\plwidth{1pt}
	\else
	\def\plwidth{#7}
	\fi
	
	\ifx\relax#9\relax
	\fill \nphedge;
	\else
	\fill[#9]\nphedge;
	\fi
	
	\ifx\relax#8\relax
	\draw[line width=\plwidth,rounded corners=\qoffs]\nphedge;
	\else
	\draw[line width=\plwidth,#8]\nphedge;
	\fi
}

\newcommand{\qedge}[7]{
	
	\ifx\relax#4\relax
	\def\qoffs{0pt}
	\else
	\def\qoffs{#4}
	\fi
	
	\def\qhedge{
		($#1+#3!\qoffs!-90:#2-#3$) --
		($#2+#1!\qoffs!-90:#3-#1$) --
		($#3+#2!\qoffs!-90:#1-#2$) -- cycle}

	\coordinate (12) at ($#1!\qoffs!90:#2$);
	\coordinate (13) at ($#1!\qoffs!-90:#3$);
	\coordinate (23) at ($#2!\qoffs!90:#3$);
	\coordinate (21) at ($#2!\qoffs!-90:#1$);
	\coordinate (31) at ($#3!\qoffs!90:#1$);
	\coordinate (32) at ($#3!\qoffs!-90:#2$);
	
	\def\nqhedge{
		(13) let \p1=($(13)-#1$), \p2=($(12)-#1$) in
		arc[start angle={atan2(\y1,\x1)}, delta angle={atan2(\y2,\x2)-atan2(\y1,\x1)-360*(atan2(\y2,\x2)-atan2(\y1,\x1)>0)}, x radius=\qoffs, y radius=\qoffs] --
		(21) let \p1=($(21)-#2$), \p2=($(23)-#2$) in
		arc[start angle={atan2(\y1,\x1)}, delta angle={atan2(\y2,\x2)-atan2(\y1,\x1)-360*(atan2(\y2,\x2)-atan2(\y1,\x1)>0)}, x radius=\qoffs, y radius=\qoffs] --
		(32) let \p1=($(32)-#3$), \p2=($(31)-#3$) in
		arc[start angle={atan2(\y1,\x1)}, delta angle={atan2(\y2,\x2)-atan2(\y1,\x1)-360*(atan2(\y2,\x2)-atan2(\y1,\x1)>0)}, x radius=\qoffs, y radius=\qoffs] --
		cycle}
	
	\ifx\relax#5\relax
	\def\qlwidth{1pt}
	\else
	\def\qlwidth{#5}
	\fi
	
	\ifx\relax#7\relax
	\fill \nqhedge;
	\else
	\fill[#7]\nqhedge;
	\fi
	
	\ifx\relax#6\relax
	\draw[line width=\qlwidth,rounded corners=\qoffs]\nqhedge;
	\else
	\draw[line width=\qlwidth,#6]\nqhedge;
	\fi
}

\newcommand{\redge}[8]{
	
	\ifx\relax#5\relax
	\def\qoffs{0pt}
	\else
	\def\qoffs{#5}
	\fi
	
	\def\rhedge{
		($#1+#4!\qoffs!-90:#2-#4$) -- 
		($#2+#1!\qoffs!-90:#3-#1$) -- 
		($#3+#2!\qoffs!-90:#4-#2$) -- 
		($#4+#3!\qoffs!-90:#1-#3$) -- cycle}

	\coordinate (12) at ($#1!\qoffs!90:#2$);
	\coordinate (14) at ($#1!\qoffs!-90:#4$);
	\coordinate (23) at ($#2!\qoffs!90:#3$);
	\coordinate (21) at ($#2!\qoffs!-90:#1$);
	\coordinate (34) at ($#3!\qoffs!90:#4$);
	\coordinate (32) at ($#3!\qoffs!-90:#2$);
	\coordinate (41) at ($#4!\qoffs!90:#1$);
	\coordinate (43) at ($#4!\qoffs!-90:#3$);
	
	\def\nrhedge{
		(14) let \p1=($(14)-#1$), \p2=($(12)-#1$) in 
		arc[start angle={atan2(\y1,\x1)}, delta angle={atan2(\y2,\x2)-atan2(\y1,\x1)-360*(atan2(\y2,\x2)-atan2(\y1,\x1)>0)}, x radius=\qoffs, y radius=\qoffs] --
		(21) let \p1=($(21)-#2$), \p2=($(23)-#2$) in 
		arc[start angle={atan2(\y1,\x1)}, delta angle={atan2(\y2,\x2)-atan2(\y1,\x1)-360*(atan2(\y2,\x2)-atan2(\y1,\x1)>0)}, x radius=\qoffs, y radius=\qoffs] --
		(32) let \p1=($(32)-#3$), \p2=($(34)-#3$) in 
		arc[start angle={atan2(\y1,\x1)}, delta angle={atan2(\y2,\x2)-atan2(\y1,\x1)-360*(atan2(\y2,\x2)-atan2(\y1,\x1)>0)}, x radius=\qoffs, y radius=\qoffs] --
		(43) let \p1=($(43)-#4$), \p2=($(41)-#4$) in 
		arc[start angle={atan2(\y1,\x1)}, delta angle={atan2(\y2,\x2)-atan2(\y1,\x1)-360*(atan2(\y2,\x2)-atan2(\y1,\x1)>0)}, x radius=\qoffs, y radius=\qoffs] --
		cycle}
	
	\ifx\relax#6\relax
	\def\rlwidth{1pt}
	\else
	\def\rlwidth{#6}
	\fi
	
	\ifx\relax#8\relax
	\fill \nrhedge;
	\else
	\fill[#8]\nrhedge;
	\fi
	
	\ifx\relax#7\relax
	\draw[line width=\rlwidth,rounded corners=\qoffs]\nrhedge;
	\else
	\draw[line width=\rlwidth,#7]\nrhedge;
	\fi
}

\let\epsilon=\varepsilon
\let\eps=\epsilon
\let\rho=\varrho
\let\theta=\vartheta

\newtheoremstyle{note}  {4pt}  {4pt}  {\sl}  {}  {\bfseries}  {.}  {.5em}          {}
\newtheoremstyle{introthms}  {3pt}  {3pt}  {\itshape}  {}  {\bfseries}  {.}  {.5em}          {\thmnote{#3}}
\newtheoremstyle{remark}  {2pt}  {2pt}  {\rm}  {}  {\bfseries}  {.}  {.3em}          {}

\theoremstyle{plain}
\newtheorem{theorem}{Theorem}[section]

\newtheorem{prop}[theorem]{Proposition}

\newtheorem{conj}[theorem]{Conjecture}

\theoremstyle{note}

\theoremstyle{remark}

\newtheorem{question}[theorem]{Question}

\usepackage{lineno}
\newcommand*\patchAmsMathEnvironmentForLineno[1]{
	\expandafter\let\csname old#1\expandafter\endcsname\csname #1\endcsname
	\expandafter\let\csname oldend#1\expandafter\endcsname\csname end#1\endcsname
	\renewenvironment{#1}
	{\linenomath\csname old#1\endcsname}
	{\csname oldend#1\endcsname\endlinenomath}}
\newcommand*\patchBothAmsMathEnvironmentsForLineno[1]{
	\patchAmsMathEnvironmentForLineno{#1}
	\patchAmsMathEnvironmentForLineno{#1*}}
\AtBeginDocument{
	\patchBothAmsMathEnvironmentsForLineno{equation}
	\patchBothAmsMathEnvironmentsForLineno{align}
	\patchBothAmsMathEnvironmentsForLineno{flalign}
	\patchBothAmsMathEnvironmentsForLineno{alignat}
	\patchBothAmsMathEnvironmentsForLineno{gather}
	\patchBothAmsMathEnvironmentsForLineno{multline}
}

\def\ex{\text{\rm ex}}

\usepackage{scalerel}

\makeatletter
\newcommand{\overrighharpoonup}[1]{\ThisStyle{%
		\vbox {\m@th\ialign{##\crcr
				\rightharpoonupfill \crcr
				\noalign{\kern-\p@\nointerlineskip}
				$\hfil\SavedStyle#1\hfil$\crcr}}}}

\def\rightharpoonupfill{%
	$\SavedStyle\m@th\mkern+0.8mu\cleaders\hbox{$\shortbar\mkern-4mu$}\hfill\rightharpoonuptip\mkern+0.8mu$}

\def\rightharpoonuptip{%
	\raisebox{\z@}[2pt][1pt]{\scalebox{0.55}{$\SavedStyle\rightharpoonup$}}}

\def\shortbar{%
	\smash{\scalebox{0.55}{$\SavedStyle\relbar$}}}
\makeatother

\makeatletter
\newcommand{\overlefharpoonup}[1]{\ThisStyle{%
		\vbox {\m@th\ialign{##\crcr
				\leftharpoonupfill \crcr
				\noalign{\kern-\p@\nointerlineskip}
				$\hfil\SavedStyle#1\hfil$\crcr}}}}

\def\leftharpoonupfill{%
	$\SavedStyle\m@th\mkern+0.8mu\cleaders\hbox{$\shortbar\mkern-4mu$}\hfill\leftharpoonuptip\mkern+0.8mu$}

\def\leftharpoonuptip{%
	\raisebox{\z@}[2pt][1pt]{\scalebox{0.55}{$\SavedStyle\leftharpoonup$}}}

\makeatletter
\newsavebox\myboxA
\newsavebox\myboxB
\newlength\mylenA

\newcommand*\xoverline[2][0.75]{%
	\sbox{\myboxA}{$\m@th#2$}%
	\setbox\myboxB\null
	\ht\myboxB=\ht\myboxA%
	\dp\myboxB=\dp\myboxA%
	\wd\myboxB=#1\wd\myboxA
	\sbox\myboxB{$\m@th\overline{\copy\myboxB}$}
	\setlength\mylenA{\the\wd\myboxA}
	\addtolength\mylenA{-\the\wd\myboxB}%
	\ifdim\wd\myboxB<\wd\myboxA%
	\rlap{\hskip 0.5\mylenA\usebox\myboxB}{\usebox\myboxA}%
	\else
	\hskip -0.5\mylenA\rlap{\usebox\myboxA}{\hskip 0.5\mylenA\usebox\myboxB}%
	\fi}
\makeatother

\begin{document}
	
	\title[The codegree Tur\'an density of $C_\ell^{-}$]
	{The codegree Tur\'an density of tight cycles minus one edge}
	
	\author[S. Piga]{Sim\'on Piga}
	\address{Mathematics Department, University of Birmingham, UK}
	\email{piga@birmingham.uk}
	
	\author[M.~Sales]{Marcelo Sales}
	\address{Mathematics Department, Emory University, USA}
	\email{mtsales@emory.edu}
	
	\author[B.~Sch\"ulke]{Bjarne Sch\"ulke}
	\address{Mathematics Department, California Institute of Technology, USA}
	\email{schuelke@caltech.edu}
	
	\thanks{
		S. Piga, is supported by EPSRC grant EP/V002279/1. The second author is partially suported by NSF grant DMS 1764385.
		There are no additional data beyond that contained within the main manuscript}
	
	
	\begin{abstract}
		Given~$\alpha>0$ and an integer~$\ell\geq5$, we prove that every sufficiently large $3$-uniform hypergraph~$H$ on $n$ vertices in which every two vertices are contained in at least~$\alpha n$ edges contains a copy of~$C_\ell^{-}$, a tight cycle on $\ell$ vertices minus one edge.
		This improves a previous result by Balogh, Clemen, and Lidick\'y.
		
	\end{abstract}
	
	\maketitle
	
	\section{Introduction}
	A $k$-uniform hypergraph $H$ consists of a vertex set~$V(H)$ together with a set of edges~$E(H)\subseteq V(H)^{(k)}=\{S\subseteq V(H):\vert S\vert =k\}$.
	Throughout this note, if not stated otherwise, by \emph{hypergraph} we always mean a $3$-uniform hypergraph.
	Given a hypergraph~$F$, the extremal number of~$F$ for~$n$ vertices, $\ex(n,F)$, is the maximum number of edges an~$n$-vertex hypergraph can have without containing a copy of~$F$. 
	Determining the value of~$\ex(n,F)$, or the Tur\'an density $\pi(F) = \lim_{n \to \infty} \frac{\ex(n,F)}{\binom{n}{3}}$, is one of the core problems in combinatorics. 
	In particular, the problem of determining the Tur\'an density of the complete~$3$-uniform hypergraph on four vertices, i.e.,~$\pi(K_4^{(3)})$, was asked by Tur\'an in 1941~\cite{T:41} and Erd\H{o}s~\cite{E:77} offered 1000\$ for its resolution.
	Despite receiving a lot of attention (see for instance the survey by Keevash~\cite{K:11}) this problem, and even the seemingly simpler problem of determining~$\pi(K_4^{(3)-})$, where~$K_4^{(3)-}$ is the~$K_4^{(3)}$ minus one edge, remain open.
	
	Several variations of this type of problem have been considered, see for instance~\cites{Balogh,GKV:16,RRS:18} and the references therein.
	The one that we are concerned with in this note asks how large the minimum codegree of an~$F$-free hypergraph can be.
	Given a hypergraph~$H$ and~$S\subseteq V$ we define the degree~$d(S)$ of~$S$ (in~$H$) as the number of edges containing~$S$, i.e.,~$d(S)=\vert\{e\in E(H):S\subseteq e\}\vert$.
	If~$S=\{v\}$ or~$S=\{u,v\}$ (and~$H$ is~$3$-uniform), we omit the parentheses and speak of~$d(v)$ or~$d(uv)$ as the degree of~$v$ or codegree of~$u$ and~$v$, respectively.
	We further write~$\delta(H)=\delta_1(H)=\min_{v\in V(H)} d(v)$ and~$\delta_2(H)=\min_{uv\in V(H)^{(2)}}d(uv)$ for the minimum degree and the minimum codegree of~$H$, respectively.
	
	Given a hypergraph~$F$ and~$n\in\mathds{N}$, Mubayi and Zhao \cite{MZ} introduced the \emph{codegree Tur\'an number}~$\ex_2(n,F)$ of~$n$ and~$F$ as the maximum~$d$ such that there is an~$F$-free hypergraph~$H$ on~$n$ vertices with~$\delta_2(H)\geq d$. 
	Moreover, they defined the \emph{codegree Tur\'an density of $F$} as
	$$\gamma(F) := \lim_{n\to\infty} \frac{ex_2(n,F)}{n}\,$$
	and proved that this limit always exists. 
	It is not hard to see that
	$$\gamma(F) \leq \pi(F)\,.$$

	The codegree Tur\'an density is known only for a few (non-trivial) hypergraphs (and blow-ups of these), see the table in~\cite{Balogh}.
	The first result that determined~$\gamma(F)$ exactly is due to Mubayi \cite{M:05} who showed that~$\gamma(\mathbb F) = 1/2$, where $\mathbb F$ denotes the `Fano plane'.
	Later, using a computer assisted proof, Falgas-Ravry, Pikhurko, Vaughan, and Volec \cite{codegK4-} proved that~$\gamma(K_4^{(3)-})=1/4$.
	As far as we know, the only other hypergraph for which the codegree Tur\'an density is known is~$F_{3,2}$, a hypergraph with vertex set~$[5]$ and edges~$123$,~$124$,~$125$, and~$345$.
	The problem of determining the codegree Tur\'an density of~$K_4^{(3)}$ remains open, and Czygrinow and Nagle~\cite{CN:01} conjectured that~$\gamma(K_4^{(3)})=1/2$.
	For more results concerning $\pi(F)$,~$\gamma(F)$, and other variations of the Tur\'an density see~\cite{Balogh}. 
	
	Given an integer~$\ell\geq 3$, a \emph{tight cycle~$C_\ell$} is a hypergraph with vertex set~$\{v_1, \dots, v_\ell\}$ and edge set~$\{v_iv_{i+1}v_{i+2}:i\in\mathds{Z}/\ell\mathds{Z}\}$.
	Moreover, we define~$C_\ell^{-}$ as~$C_\ell^{}$ minus one edge. 
	In this note we prove that the Tur\'an codegree density of~$C_\ell^{-}$ is zero for every~$\ell\geq 5$. 
	
	\begin{theorem}\label{thm:cycles}
		Let~$\ell\geq 5$ be an integer.
		Then~$\gamma(C_{\ell}^{-})=0$.
	\end{theorem}
	The previously known best upper bound was given by Balogh, Clemen, and Lidick\'y \cite{Balogh} who used flag algebras to prove that~$\gamma(C_{\ell}^-)\leq 0.136$. 
	
	\section{Proof of Theorem~\ref{thm:cycles}}

	For singletons, pairs, and triples we may omit the set parentheses and commas.
	For a hypergraph~$H=(V,E)$ and~$v\in V$, the \emph{link of~$v$} (in~$H$) is the graph~$L_v=(V\setminus v,\{e\setminus v:v\in e\in E\})$.
	For~$x,y\in V$, the neighbourhood of~$x$ and~$y$ (in~$H$) is the set~$N(xy)=\{z\in V:xyz\in E\}$.
	For positive integers~$\ell, k$ and a hypergraph $F$ on $k$ vertices, denote the \emph{$\ell$-blow-up of~$F$} by $F(\ell)$.
	This is the $k$-partite hypergraph $F(\ell)=(V, E)$ with $V = V_1 \dot\cup \dots \dot\cup V_k$, $|V_i| = \ell$ for $1\leq i \leq k$, and $E = \{v_{i_1}v_{i_2} v_{i_3}: v_{i_j} \in V_{i_j} \text{ and } i_1 i_2 i_3 \in E(F)\}$.

	In their seminal paper, Mubayi and Zhao \cite{MZ} proved the following supersaturation result for the codegree Tur\'an density.
	
	\begin{prop}[Mubayi and Zhao \cite{MZ}]\label{prop:blow-up}
		For every hypergraph~$F$ and~$\varepsilon>0$, there are~$n_0$ and~$\delta>0$ such that every hypergraph~$H$ on~$n\geq n_0$ vertices with~$\delta_2(H)\geq (\gamma(F)+\varepsilon)n$ contains at least~$\delta n^{v(F)}$ copies of~$F$.
		Consequently, for every positive integer~$\ell$, $\gamma (F) = \gamma(F(\ell))$.\qed 
	\end{prop}

	\begin{proof}[Proof of Theorem~\ref{thm:cycles}]
		
		We begin by noting that it is enough to show that~$\gamma(C_5^-)\!=\!0$. Indeed, we shall prove by induction that~$\gamma(C_{\ell}^-)=0$ for every $\ell\geq 5$. For~$\ell=6$, the result follows since~$C_6^-$ is a subgraph of~$C_3(2)$. Hence, by Proposition~\ref{prop:blow-up}, we have~$\gamma(C_6^-)\leq\gamma(C_3(2))=\gamma(C_3)=0$. For~$\ell=7$, note that~$C_7^-$ is a subgraph of~$C_5^-(2)$. To see that, let~$v_1,\dots,v_5$ be the vertices of a~$C_5^-$ with edge set~$\{v_iv_{i+1}v_{i+2}:i\neq 4\}$, where the indices are taken modulo $5$. Now add one copy~$v_2'$ of~$v_2$ and one copy~$v_3'$ of~$v_3$.
		Then~$v_1v_3v_2v_4v_3'v_5v_2'$ is the cyclic ordering of a~$C_7^-$ with the missing edge being~$v_3'v_5v_2'$.
		Therefore, if~$\gamma(C_5^-)=0$, then, by Proposition~\ref{prop:blow-up}, we have~$\gamma(C_7^-)=0$. Finally, for~$\ell\geq 8$,~$\gamma(C_\ell^-)=0$ follows by induction using the same argument and observing that $C_\ell^-$ is a subgraph of $C_{\ell-3}^-(2)$.
		
		
		
		Given $\varepsilon\in (0,1)$, consider a hypergraph~$H=(V,E)$ on~$n\geq \big(\frac{2}{\varepsilon}\big)^{5/\varepsilon^2+2}$ vertices with~$\delta_2(H)\geq\varepsilon n$. 
		We claim that $H$ contains a copy of a $C_5^-$.
		
		Given~$v,b\in V$,~$S\subseteq V$, and~$P\subseteq (V\setminus S)^{2}$, we say that~$(v,S,b,P)$ is a \textit{nice picture} if it satisfies the following:
		\begin{enumerate}
			\item[(i)] $S\subseteq N_{L_v}(b)$, where~$N_{L_v}(b)$ is the neighborhood of~$b$ in the link~$L_v$.
			\item[(ii)] For every vertex~$u\in S$ and ordered pair~$(x,y)\in P$, the sequence~$ubxy$ is a path of length~$3$ in~$L_v$.
		\end{enumerate}
		Note that if~$(v,S,b,P)$ is a nice picture and there exists~$u\in S$ and~$(x,y)\in P$ such that~$uxy\in E$, then~$ubvxy$ is a copy of~$C_5^-$ (with the missing edge being~$yub$)
		
		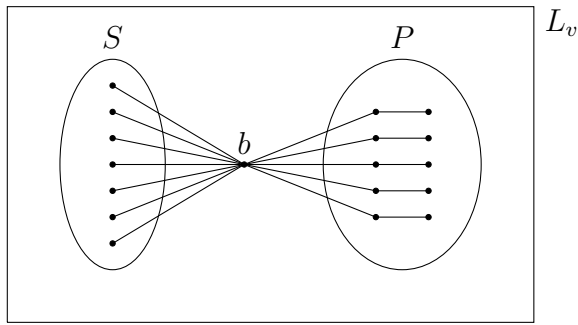
\begin{figure}[h]
			\centering
			{\hfil \begin{tikzpicture}[scale=0.7]
					
					\draw (2,3) ellipse (1 and 2);
					\draw (7.5,3) ellipse (1.5 and 2);

					\coordinate (P) at (0,0);
					\coordinate (Q) at (10,0);
					\coordinate (R) at (10,6);
					\coordinate (S) at (0,6);
					
					\coordinate (b) at (4.5,3);
					\coordinate (z1) at (2,4.5);
					\coordinate (z2) at (2,4);
					\coordinate (z3) at (2,3.5);
					\coordinate (z4) at (2,3);
					\coordinate (z5) at (2,2.5);
					\coordinate (z6) at (2,2);
					\coordinate (z7) at (2,1.5);
					\coordinate (u) at (2,5);
					\coordinate (d) at (2,1);
					
					\coordinate (x1) at (7,4);
					\coordinate (x2) at (7,3.5);
					\coordinate (x3) at (7,3);
					\coordinate (x4) at (7,2.5);
					\coordinate (x5) at (7,2);
					\coordinate (y1) at (8,4);
					\coordinate (y2) at (8,3.5);
					\coordinate (y3) at (8,3);
					\coordinate (y4) at (8,2.5);
					\coordinate (y5) at (8,2);
					
					\draw[fill] (b) circle [radius=0.05];
					\draw[fill] (z1) circle [radius=0.05];
					\draw[fill] (z2) circle [radius=0.05];
					\draw[fill] (z3) circle [radius=0.05];
					\draw[fill] (z4) circle [radius=0.05];
					\draw[fill] (z5) circle [radius=0.05];
					\draw[fill] (z6) circle [radius=0.05];
					\draw[fill] (z7) circle [radius=0.05];
					\draw[fill] (x1) circle [radius=0.05];
					\draw[fill] (x2) circle [radius=0.05];
					\draw[fill] (x3) circle [radius=0.05];
					\draw[fill] (x4) circle [radius=0.05];
					\draw[fill] (x5) circle [radius=0.05];
					\draw[fill] (y1) circle [radius=0.05];
					\draw[fill] (y2) circle [radius=0.05];
					\draw[fill] (y3) circle [radius=0.05];
					\draw[fill] (y4) circle [radius=0.05];
					\draw[fill] (y5) circle [radius=0.05];

					\draw (P)--(Q)--(R)--(S)--(P);
					\draw (b)--(z1);
					\draw (b)--(z2);
					\draw (b)--(z3);
					\draw (b)--(z4);
					\draw (b)--(z5);
					\draw (b)--(z6);
					\draw (b)--(z7);
					\draw (b)--(x1);
					\draw (b)--(x2);
					\draw (b)--(x3);
					\draw (b)--(x4);
					\draw (b)--(x5);
					\draw (x1)--(y1);
					\draw (x2)--(y2);
					\draw (x3)--(y3);
					\draw (x4)--(y4);
					\draw (x5)--(y5);
					
					\node (b) at (b) [above] {$b$};
					\node (s) at (2,5) [above] {$S$};
					\node (P) at (7.5,5) [above] {$P$};
					\node (L) at (10,5.7) [right] {$L_v$};

				\end{tikzpicture}\hfil}
			\caption{A nice picture $(v,S,b,P)$}
		\end{figure}
		
		To find such a copy of~$C_5^-$ in~$H$, we are going to construct a sequence of nested sets~$S_t\subseteq S_{t-1}\subseteq \ldots \subseteq S_0$, where~$t=5/\varepsilon^2+1$, and nice pictures~$(v_i,S_i,b_i,P_i)$ satisfying~$v_i\in S_{i-1}$,~$|S_i|\geq \big(\frac{\eps}{2}\big)^{i+1}n\geq 1$ and~$|P_i|\geq \varepsilon^2 n^2/5$ for~$1\leq i \leq t$. 
		Suppose that such a sequence exists. 
		Then by the pigeonhole principle, there exist two indices~$i,j \in [t]$ such that~$P_i\cap P_j\neq \emptyset$ and~$i<j$. Let~$(x,y)$ be an element of~$P_i\cap P_j$. Hence, we obtain a nice picture~$(v_i,S_i,b_i,P_i)$,~$v_j\in S_i$ and~$(x,y)\in P_i$ such that~$v_jxy \in E$ (since~$xy$ is an edge in~$L_{v_j}$). Consequently,~$v_jb_iv_ixy$ is a copy of~$C_5^-$ in~$H$. 
		
		It remains to prove that the sequence described above always exists. 
		We construct it recursively. Let $S_0\subseteq V$ be an arbitrary subset of size $\eps n/2$. Suppose that we already constructed nice pictures $(v_i,S_i,b_i,P_i)$ for $1\leq i< k\leq t$ and now we want to construct $(v_k,S_k,b_k,P_k)$. Pick $v_k\in S_{k-1}$ arbitrarily. The minimum codegree of~$H$ implies that~$\delta(L_{v_k})\geq\varepsilon n$ and thus for every~$u\in S_{k-1}$, we have that~$d_{L_{v_k}}(u)\geq\varepsilon n$. Observe that
		\begin{align*}
			\sum_{b \in V\setminus v_k} \vert N_{L_{v_k}}( b)\cap S_{k-1}\vert 
			= \sum_{u \in S_{k-1}\setminus v_k} d_{L_{v_k}}(u)
			\geq \eps n \big(\vert S_{k-1}\vert-1\big) 
			\geq \Big(\frac{\varepsilon}{2}\Big)^{k+1}n^2
		\end{align*}
		and therefore, by an averaging argument there is a vertex~$b_k\in V\setminus v_k$ such that the subset~$S_k:=N_{L_{v_k}}(b_k)\cap S_{k-1}\subseteq S_{k-1}$ is of size at least~$|S_k|\geq \big(\frac{\eps}{2}\big)^{k+1} n$.
		Let~$P_k$ be all the pairs~$(x,y) \in (V\setminus S_k)^2$ such that for every vertex~$v\in S_k$, the sequence~$v,b_k,x,y$ forms a path of length~$3$ in~$L_{v_k}$.
		Since~$|S_k|\leq \eps n/2$ and~$\delta(L_{v_k})\geq\varepsilon n$, it is easy to see that~$|P_k|\geq \varepsilon^2n^2/5$.
		That is to say~$(v_k, S_k, b_k, P_k)$ is a nice picture satisfying the desired conditions.
	\end{proof}

	\section{Concluding remarks}
	
	A famous result by Erd\H{o}s \cite{E:64} asserts that a hypergraph $F$ satisfies $\pi(F)=0$ if $F$ is tripartite (i.e., $V(F)=X_1\dot\cup X_2 \dot\cup X_3$ and for every~$e\in E(F)$ we have~$\vert e\cap X_i\vert = 1$ for every~$i\in[3]$).
	Note that if~$H$ is tripartite, then every subgraph of $H$ is tripartite as well and there are tripartite hypergraphs $H$ with~$\vert E(H) \vert =\tfrac{2}{9} \binom{\vert V(H)\vert }{3}$. 
	Therefore, if $F$ is not tripartite, then~$\pi(F)\geq 2/9$.
	In other words, Erd\H{o}s' result implies that there are no Tur\'an densities in the interval~$(0,2/9)$.
	It would be interesting to understand the behaviour of the codegree Tur\'an density in the range close to zero.
	
	\begin{question}
		Is it true that for every~$\xi\in (0,1]$, there exists a~hypergraph~$F$ such that \vspace{-0.3cm}
		\[
		0<\gamma(F)\leq \xi\text{ ?}
		\]
	\end{question}
	
	Mubayi and Zhao \cite{MZ} answered this question affirmatively if we consider the codegree Tur\'an density of a family of hypergraphs instead of a single hypergraph. 
	
	Since $C_5^{-}$ is not tripartite, we have that~$\pi(C_5^{-})\geq 2/9$. 
	The following construction attributed to~Mubayi and R\"odl (see e.g.~\cites{Balogh}) provides a better lower bound. 
	Let~$H=(V,E)$ be a~$C_5^{-}$-free hypergraph on~$n$ vertices.
	Define a hypergraph~$\widetilde H$ on~$3n$ vertices with~$V(\widetilde H)= V_1\dot\cup V_2\dot\cup V_3$ such that~$\widetilde H[V_i]=H$ for every~$i\in [3]$ plus all edges of the form~$e=\{v_1,v_2,v_3\}$ with $v_i\in V_i$.
	Then, it is easy to check that $\widetilde H$ is also~$C_5^{-}$-free. 
	We may recursively repeat this construction starting with~$H$ being a single edge and obtain an arbitrarily large~$C_5^{-}$-free hypergraph with density~$1/4-o(1)$. 
	In fact, those hypergraphs are~$C_\ell^-$-free for every~$\ell$ not divisible by three.
	The following is a generalisation of a conjecture in~\cite{MPS:11}.
	
	
	\begin{conj} If~$\ell\geq 5$ is not divisible by three, then 
		$\pi(C_\ell^{-}) = \frac{1}{4}\,.$
	\end{conj}

\end{document}